\documentclass{amsart}

\newtheorem{theorem}{Theorem}

\usepackage{amsmath,amssymb}

\title{On a partition problem of Canfield and Wilf}

\author{\v{Z}eljka Ljuji{\' c}}
\address{CUNY Graduate Center,
New York, NY 10016}
\email{zeljka.ljujic@gmail.edu}

\author{Melvyn B.  Nathanson}
\address{Lehman College (CUNY), 
Bronx, NY 10468, and 
CUNY Graduate Center,
New York, NY 10016}
\email{melvyn.nathanson@lehman.cuny.edu}

\dedicatory{To the memory of John Selfridge}


\thanks{The work of M.B.N. was supported in part by a PSC-CUNY Research Award.}

\date{\today}

\begin{document}

\begin{abstract}
Let $A$ and $M$ be nonempty sets of positive integers.
A partition of the positive integer $n$ with parts in $A$ and multiplicities in $M$ is a representation of $n$ in the form 
$n = \sum_{a\in A} m_a a$
where $m_a \in M \cup \{0\}$  for all $a\in A$, and $m_a \in M$ for only finitely many $a$.
Denote by $p_{A,M}(n)$  the number of partitions of $n$ with parts in $A$ and multiplicities in $M$. 
It is proved that there exist infinite sets $A$ and $M$ of positive integers whose partition function $p_{A,M}$ has weakly superpolynomial but not superpolynomial growth.  
The counting function of the set $A$ is $A(x) = \sum_{a \in A, a\leq x} 1$. 
 It is also proved that $p_{A,M}$ must have at least weakly superpolynomial growth if $M$ is infinite and $A(x) \gg \log x$.
\end{abstract}

\maketitle

\section{Partition problems with restricted multiplicities}  \label{section:intro}

Let $A$ be a nonempty set of positive integers.
A \emph{partition of $n$ with parts in $A$} is a representation of $n$ in the form 
\[
n = \sum_{a\in A} m_a a
\]
where $m_a \in \mathbf{N} \cup \{0\}$  for all $a\in A$, and $m_a \in \mathbf{N}$ for only finitely many $a$.
Let $p_A(n)$ denote the number of partitions of $n$ with parts in $A$. 
If $\gcd(A)= d  > 1$,  then $p_A(n)=0$ for all $n$ not divisible 
by $d$, and so $p_A(n) = 0$ for infinitely many 
positive integers $n$.  
If $p_A(n) \geq 1$ for all sufficiently large $n$, then $\gcd(A) = 1.$

If $A = \{a_1,\ldots, a_k\}$ is a  set of $k$ relatively prime positive integers, then Schur~\cite{schu26} proved that
\begin{equation}  \label{LN:Schur}
p_A(n) \sim \frac{n^{k-1}}{(k-1)!a_1a_2\cdots a_k}.
\end{equation}
Nathanson~\cite{nath00a} gave a simpler proof of the more precise result:
\begin{equation}  \label{LN:Schur-MBN}
p_A(n) = \frac{n^{k-1}}{(k-1)!a_1a_2\cdots a_k} + O\left(n^{k-2}\right).
\end{equation}

An arithmetic function is a real-valued function whose domain is the set $\mathbf{N}$ of positive integers.  
An arithmetic function $f$ has \emph{polynomial growth} 
if there is a positive integer $k$ and an integer $N_0(k)$ such that $f(n) \leq n^k$ for all  $n \geq N_0(k)$.  
Equivalently, $f$ has polynomial growth if 
\[
\limsup_{n\rightarrow \infty} \frac{\log f(n)}{\log n} < \infty.
\]
An arithmetic function $f$ has \emph{superpolynomial growth} 
if 
\[
\lim_{n\rightarrow \infty} \frac{\log f(n)}{\log n} = \infty.
\]

The asymptotic formula~\eqref{LN:Schur} implies the following result of  Nathanson~\cite[Theorem 15.2, pp. 458--461]{nath00aa}.

\begin{theorem}
If $A$ is an infinite set of integers and $\gcd(A) = 1$, 
then $p_A(n)$ has superpolynomial growth.   
\end{theorem}

Canfield and Wilf~\cite{canf-wilf10} studied the following variation of the classical partition problem.  
Let $A$ and $M$ be nonempty sets of positive integers.
A \emph{partition of $n$ with parts in $A$ and multiplicities in $M$} is a representation of $n$ in the form 
\[
n = \sum_{a\in A} m_a a
\]
where $m_a \in M \cup \{0\}$  for all $a\in A$, and $m_a \in M$ for only finitely many $a$.
We denote by $p_{A,M}(n)$  the number of partitions of $n$ with parts in $A$ and multiplicities in $M$. 
Note that $p_{A,M}(0)=1$ and $p_{A,M}(n)=0$ for all $n < 0$.  

Let $A$ and $M$ be infinite sets of positive integers such that 
$p_{A,M}(n) \geq 1$ for all sufficiently large $n$.
Canfield and Wilf (``Unsolved problem 1'' in~\cite{canf-wilf10}) asked, ``Must $p_{A,M}(N)$ then be of superpolynomial growth?''
We prove that the answer is ``no.''

\section{Weakly superpolynomial functions}   \label{section:weaklysuper}

Polynomial and superpolynomial growth functions were first studied in connection with the growth of finitely and infinitely generated groups (cf. Grigorchuk and Pak~\cite{grig-pak08}, Nathanson~\cite{nath10d}).  
Growth functions of groups are always strictly increasing, but even strictly increasing functions that do not have polynomial growth are not necessarily superpolynomial.  

We shall call an arithmetic function \emph{weakly superpolynomial} if it does not have polynomial growth.
Equivalently, the function $f$ is weakly superpolynomial if for every positive integer $k$ there are infinitely many positive integers $n$ such that $f(n) > n^k$.  
The partition functions that will be constructed  in this paper are weakly superpolynomial but not superpolynomial.

We note that an arithmetic function $f$ is weakly superpolynomial but not superpolynomial if and only if 
\[
\limsup_{n\rightarrow \infty} \frac{\log f(n)}{\log n} = \infty
\]
and
\[
\liminf_{n\rightarrow \infty} \frac{\log f(n)}{\log n} < \infty.
\]
In this section we construct a strictly increasing arithmetic function that is weakly superpolynomial but not polynomial.

Let $(n_k)_{k=1}^{\infty}$ be a  sequence of positive integers such that $n_1 = 1$ and 
\[
n_{k+1} > 2 n_k^k
\]
for all $k \geq 1$.
We define the arithmetic function
\[
f(n) = n_k^k + (n-n_k) \qquad\text{for $n_k \leq n < n_{k+1}$.}
\]
This function is strictly increasing because
\[
n_k^k - n_k \leq n_{k+1}^{k+1} - n_{k+1}
\]
for all $k \geq 1$.
We have  
\[
\lim_{k\rightarrow\infty} \frac{\log f(n_k)}{\log n_k} 
= \lim_{k\rightarrow\infty} \frac{k\log n_k}{\log n_k} = \infty
\]
and so 
\[
\limsup_{n\rightarrow\infty} \frac{\log f(n)}{\log n} = \infty.
\]
Therefore, the function $f$ does not have polynomial growth.

For every positive integer $n$ there is a positive integer $k$ such that 
$n_k \leq n < n_{k+1}$. 
Then $f(n) = n + n_k^k - n_k \geq n$ and so 
\begin{equation}   \label{LN:inf1}
\liminf_{n\rightarrow\infty} \frac{\log f(n)}{\log n} \geq 1.
\end{equation}
The inequalities 
\[
 f(n_{k+1}-1) = n_k^k + (n_{k+1} -1- n_k) < \frac{3n_{k+1}}{2}
\]
and
\[
n_{k+1}-1 >  \frac{n_{k+1}}{2}
\]
imply that 
\[
1 <  \frac{\log f(n_{k+1}-1)}{\log (n_{k+1}-1)} 
< \frac{\log ( 3n_{k+1}/2 )}{\log (n_{k+1}/2)} 
= 1 + \frac{\log 3}{\log (n_{k+1}/2)}
\]
and so 
\[
\lim_{k\rightarrow\infty}  \frac{\log f(n_{k+1}-1)}{\log (n_{k+1}-1)} = 1.
\]
Therefore,
\begin{equation}   \label{LN:inf2}
\liminf_{n\rightarrow\infty} \frac{\log f(n)}{\log n} \leq 1.
\end{equation}
Combining~\eqref{LN:inf1} and~\eqref{LN:inf2}, we obtain 
\[
\liminf_{n\rightarrow\infty} \frac{\log f(n)}{\log n} = 1.
\]  
Thus, the function $f$ has weakly superpolynomial but not superpolynomial growth.

\section{Weakly superpolynomial partition functions}    \label{section:partitions}

\begin{theorem}      \label{LN:theorem:AM}
Let $a$ be an integer, $a \geq 2$,  and let $A = \{a^i\}_{i=0}^{\infty}$.
Let $M$ be an infinite set of positive integers such that $M$ contains $ \{1,2,\ldots, a-1\}$ and  no element of $M$ is divisible by $a$.  
Then $p_{A,M}(n) \geq 1$ for all $n \in \mathbf{N}$, and $p_{A,M}\left(n\right) = 1$ for all $n \in A$.
In particular, the partition function 
$p_{A,M}$ does not have  superpolynomial growth.
\end{theorem}

\begin{proof}
Every positive integer $n$ has a unique $a$-adic representation, and so $p_{A,M}(n) \geq 1$ for all $n \in \mathbf{N}$.

We shall prove that the only partition of $a^r$ with parts in $A$ and multiplicities in $M$ is $a^r = 1 \cdot a^r.$  
If there were another representation, then it could be written in the form  
\[
a^r = \sum_{i=1}^k m_i a^{j_i}
\]
where $k \geq 2$, $m_i \in M$ for $i=1,\ldots, k$,  and
$0 \leq j_1 < j_2 < \cdots < j_k < r$.
Then 
\[
a^{r-j_1} = m_1 + a\sum_{i=2}^k m_ia^{j_i-j_1-1}.
\]
We have $j_i-j_1-1\geq 0$ for $i=2,\ldots, k$, and so $m_1$ is divisible by $a$, which is absurd.
Therefore, $p_{A,M}\left(a^r\right) = 1$ for all $r \geq 0$. 
It follows that 
\[
\liminf_{n\rightarrow\infty} \frac{\log p_{A,M}(n)}{\log n} 
= \liminf_{r\rightarrow\infty}
\frac{\log p_{A,M}\left(a^r \right)}{\log a^r} = 0
\]
and so the partition function $p_{A,M}$ is not superpolynomial.
\end{proof}

\begin{theorem}    \label{LN:theorem:WeaklySuper}
Let $A$ and $M$ be infinite sets of positive integers,
and let $p_{A,M}(n)$ denote the number of partitions of $n$ with parts in $A$ and multiplicities in $M$.
If $A(x) \geq c\log x$ for some $c > 0$ and all $x\geq x_0(A)$, then for every positive integer $k$ there exist infinitely many integers $n$ such that 
\[
p_{A,M}(n) \geq n^k.
\]
In particular, the partition function $p_{A,M}$ is weakly superpolynomial.
\end{theorem}

\begin{proof}
Let $x \geq 1$ and let 
\[
A(x) = \sum_{\substack{a\in A \\a \leq x}}1
\qquad\text{and}\qquad
M(x) = \sum_{\substack{m\in M \\m \leq x}}1
\]
denote the counting functions of the sets $A$ and $M$, respectively.
If $n \leq x$ and $n = \sum_{a\in A}m_a a$ is a partition of $n$ 
with parts in $A$ and multiplicities in $M$, 
then $a \leq x$ and $m_a \leq x$, and so
\begin{equation}  \label{LN:ineq}
\max\left\{p_{A,M}(n) : n \leq x  \right\} \leq \sum_{n\leq x} p_{A,M}(n) \leq (M(x)+1)^{A(x)}.
\end{equation}
Conversely, if the integer $n$ can be represented in the form 
$n = \sum_{a\in A}m_a a$ with $a \leq x$ and $m_a \leq x$, 
then $n \leq x^2A(x)$ and so
\[
\sum_{n\leq x^2A(x)} p_{A,M}(n) \geq (M(x)+1)^{A(x)} > 
M(x)^{A(x)}.
\]
Choose an integer $n_x$ such that 
\[
p_{A,M}(n_x) = \max\left\{ p_{A,M}(n) : n \leq x^2A(x) \right\}.
\]
Then $n_x \leq x^3$ and 
\[
M(x)^{A(x)} < \sum_{n\leq x^2A(x)} p_{A,M}(n) 
\leq \left( x^2A(x) + 1 \right) p_{A,M}(n_x)
\leq 2x^3 p_{A,M}(n_x)
\]
and so, for all $x \geq x_0(A)$, 
\[
p_{A,M}(n_x) > \frac{M(x)^{A(x)}}{2x^3} 
\geq \frac{M(x)^{c\log x}}{2x^3}.
\]
Let $k$ be a positive integer.  
Because the set $M$ is infinite, there exists $x_1(A,k) \geq x_0(A)$ such that, for all $x\geq x_1(A,k)$, we have 
\[
\log M(x) > \frac{\log 2}{c\log x} + \frac{3k+3}{c}
\]
and so 
\[
p_{A,M}(n_x) > x^{3k} \geq n_x^k.
\]

We shall iterate this process to construct inductively an infinite set of integers $\{n_{x_i}:i=1,2,\ldots \}$ such that 
\[
p_{A,M}(n_{x_i}) > n_{x_i}^k
\]
for all $i$.  
Let $r \geq 1$, and suppose that the integers 
$n_{x_1}, \ldots, n_{x_r}$ have been constructed.  
Choose $x_{r+1}$ so that 
\[
x_{r+1}^{3k} > \left(  M\left(x_i^2 A(x_i)\right) + 1 \right)^{A\left( x_i^2A(x_i) \right)}
\]
for all $i=1,\ldots, r$, and let $n_{x_{r+1}}$ be the integer constructed according to procedure above.  
Replacing $x$ with $x_i^2A(x_i)$ in inequality~\eqref{LN:ineq}, 
we see that 
\[
p\left(n_{x_i}\right) \leq \left(  M\left(x_i^2 A(x_i)\right) + 1 \right)^{A\left( x_i^2A(x_i)\right)}
\]
and so 
\[
p\left(n_{x_{r+1}} \right) > x_{r+1}^{3k}  > p\left(n_{x_i} \right) 
\]
for $i=1,\ldots, r$, and so $n_{x_{r+1}}  \neq n_{x_i}$ for $i=1,\ldots, r$.  
This completes the induction and the proof.
\end{proof}

\begin{theorem}
The partition function for the sets $A$ and $M$ constructed in Theorem~\ref{LN:theorem:AM} is weakly superpolynomial.
\end{theorem}

\begin{proof}
For $a \geq 2$, the  counting function for the set $A = \{a^i\}_{i=1}^{\infty}$ is $A(x) = [\log x]+1 > \log x$, and the result follows from Theorem~\ref{LN:theorem:WeaklySuper}.
\end{proof}

\section{Open problems}      \label{section:problems}
 
\begin{enumerate}
\item
Do there exist infinite sets $A$ and $M$ of positive integers such that $p_{A,M}(n) \geq 1$ for all sufficiently large $n$ and 
$p_{A,M}(n)$ has polynomial growth?  
This is, perhaps, the intended statement of Canfield and Wilf's "Unsolved problem 1."  

\item
By Theorem~\ref{LN:theorem:WeaklySuper}, 
if the partition function $p_{A,M}$ has polynomial growth, then the set $A$ must have sub-logarithmic growth, that is, $A(x) \gg \log x$ is impossible.   
\begin{enumerate}
\item
Let $A = \{k!\}_{k=1}^{\infty}$.  
Does there exist an infinite set $M$ of positive integers such that $p_{A,M}(n) \geq 1$ for all sufficiently large $n$ and $p_{A,M}$ has polynomial growth?
\item
Let $A = \left\{k^k\right\}_{k=1}^{\infty}$.  
Does there exist an infinite set $M$ of positive integers such that $p_{A,M}(n) \geq 1$ for all sufficiently large $n$ and $p_{A,M}$ has polynomial growth?
\end{enumerate}

\item
Let $A$ be an infinite set of positive integers 
and let $M=\mathbf{N}$.  
Bateman and Erd\H os~\cite{bate-erdo56}  proved that 
the partition function $p_A = p_{A,\mathbf{N}}$ is eventually strictly increasing if and only if $\gcd(A\setminus \{a\})=1$ for all $a\in A$.
It would be interesting to extend this result to partition functions with restricted multiplicities:  Determine a necessary and sufficient condition for infinite sets $A$ and $M$ of positive integers to have the property that $p_{A,M}(n) < p_{A,M}(n+1)$ 
or $p_{A,M}(n) \leq p_{A,M}(n+1)$ for all sufficiently large $n$.

\end{enumerate}

\def\cprime{$'$} \def\cprime{$'$} \def\cprime{$'$}
\providecommand{\bysame}{\leavevmode\hbox to3em{\hrulefill}\thinspace}
\providecommand{\MR}{\relax\ifhmode\unskip\space\fi MR }
\providecommand{\MRhref}[2]{%
  \href{http://www.ams.org/mathscinet-getitem?mr=#1}{#2}
}
\providecommand{\href}[2]{#2}

\end{document}